\documentclass[a4paper,leqno,11pt]{article}
\usepackage{amsmath, amsthm, amssymb}
\usepackage{mathrsfs}

\newtheorem{thm}{Theorem}[section]
\newtheorem{cor}[thm]{Corollary}

\newtheorem{prop}[thm]{Proposition}

\newtheorem{tthm}{Theorem}[subsection]
\newtheorem{ccor}[tthm]{Corollary}
\newtheorem{llem}[tthm]{Lemma}

\theoremstyle{definition}
\newtheorem{defin}[thm]{Definition}
\newtheorem{rem}[thm]{Remark}
\newtheorem{prob}[thm]{Problem}
\newtheorem{exam}[thm]{Example}

\makeatletter
\@addtoreset{equation}{section} 
\def\theequation{\thesection.\arabic{equation}}
\makeatother

\begin{document}

%
%  -----------------------------
%       << Newcommand >>
%  -----------------------------
%

\newcommand{\BC}{\mathbb{C}}
\newcommand{\BR}{\mathbb{R}}
\newcommand{\BN}{\mathbb{N}}
\newcommand{\BZ}{\mathbb{Z}}

%
%  -------------------------------------------
%           << Title and Author >>
%  -------------------------------------------
%

\title{On a class of singular nonlinear first order   
           partial differential equations}

\author{{} \\
Hidetoshi \textsc{TAHARA}\footnote{%Department of Information and 
        %Communication Sciences, 
        Sophia University, Tokyo 
        102-8554, Japan.
        E-mail: \texttt{h-tahara@sophia.ac.jp}}}
\date{}

\maketitle

\begin{abstract}
   In this paper, we consider a class of singular nonlinear first 
order partial differential equations 
$t(\partial u/\partial t)=F(t,x,u, \partial u/\partial x)$ with
$(t,x) \in \BR \times \BC$ under the assumption that $F(t,x,z_1,z_2)$
is a function which is continuous in $t$ and holomorphic in the other
variables. Under suitable conditions, we determine all the solutions
of this equation in a neighborhood of the origin.
\end{abstract}



\renewcommand{\thefootnote}{\fnsymbol{footnote}}

%
%  -----------------------------------------------------
%              << \S 1. Introduction >>
%  -----------------------------------------------------
%

\section{Introduction}\label{section1}
%{\bf \S 1. Introduction.}
%

   The ordinary differential equation
\begin{equation}\label{1.1}
      t \,\frac{du}{dt} 
              = f(t,u) \quad \mbox{under $f(0,0)=0$}
\end{equation}
is called the Briot-Bouquet's ordinary differential equation.
This was first studied by Briot-Bouquet \cite{bb}. In the analytic
case, we have many references: Hille \cite{hille}, 
Hukuhara-Kimura-Matuda \cite{HKM}, G\'erard \cite{gerard}, 
Iwasaki-Kimura-Shimomura-Yoshida \cite{gauss}, etc. 
In the non-analytic case, we have some references:
Wintner \cite{wintner}, Sansone-Conti \cite{sansone}.
\par
   In this paper, as a generalization of (\ref{1.1}) to partial 
differential equations, we will consider a class of singular nonlinear 
first order partial differential equations of the form
\begin{equation}\label{1.2}
   t \,\frac{\partial u}{\partial t}
   = F \Bigl( t,x,u, \frac{\partial u}{\partial x} \Bigr),
\end{equation}
and determine all the solutions (including solutions with singularity 
at $t=0$) of the equation.
\par
   Let $t \in \BR$, $x \in \BC$, $z_1 \in \BC$ and $z_2 \in \BC$ be 
the variables.  For $r>0$ we write $D_r=\{y \in \BC \,;\, |y|<r \}$,
where $y$ represents $x$, $z_1$ or $z_2$. Let $T_0>0$, $R_0>0$, 
$\rho_0>0$, and set
$\Omega=\{(t,x,z_1,z_2) \in [0,T_0] \times D_{R_0} \times D_{\rho_0} 
\times D_{\rho_0} \}$.
\par
   Let $F(t,x,z_1,z_2)$ be a function on $\Omega$, and let us consider
the equation (\ref{1.2}) under the following assumptions:
\begin{description}
\item[\quad \,\,${\rm A}_1)$] $F(t,x,z_1,z_2)$ is a continuous function 
   on $\Omega$ which is holomorphic in \hspace*{4mm}the variables
   $(x,z_1,z_2) \in D_{R_0} \times D_{\rho_0} \times D_{\rho_0}$ for 
   any fixed $t$.
\item[\quad \,\,${\rm A}_2)$] $F(0,x,0,0)=0$ on $D_{R_0}$.
\item[\quad \,\,${\rm A}_3)$] $\dfrac{\partial F}{\partial z_2}
               (0,x,0,0)=0$ on $D_{R_0}$.
\end{description}
\par
   In the case where $F(t,x,z_1,z_2)$ is a holomorphic function in 
a neighborhood of $\BC_t \times \BC_x \times \BC_{z_1} \times \BC_{z_2}$,
this equation is called a Briot-Bouquet type partial differential
equation in $t$, and it was investigated in details by G\'erard-Tahara
\cite{gt1,book} and Yamazawa \cite{yamazawa}. See also Li \cite{Li} and
Yamazawa \cite{qyamazawa}.
\par
   We set
\[
         \lambda(t,x)= \frac{\partial F}{\partial z_1}(t,x,0,0).
\]
In the case where $F(t,x,z_1,z_2)$ is only continuous in $t$, the 
existence and the uniqueness of the solution were establised 
under the condition ${\rm Re} \lambda(0,0)<0$ by Lope-Roque-Tahara 
\cite{LRT} and Tahara \cite{unique}. This concludes that in the
case ${\rm Re} \lambda(0,0)<0$ the set of all the solutions
of (\ref{1.2}) consists of one solution.
\par
   But, in the case ${\rm Re} \lambda(0,0)>0$ we have no results
up to now. Thus, in this paper we will consider the following problem:

\begin{prob}\label{Problem1.1}
    Determine all the solutions of (\ref{1.2}) under
${\rm A}_1) \sim {\rm A}_3)$ and
\begin{equation}\label{1.3}
      {\rm Re} \lambda(0,0)>0.
\end{equation}
\end{prob}

   It should be noted that the $t$-continuous version of 
Cauchy-Kowalevsky theorem was first proved by Nagumo \cite{nagumo}, 
and it was re-proved by Treves \cite{treves}, Nirenberg \cite{nirenberg} 
in a different way. Since then, to extend a holomorphic version 
to a $t$-continuous version has been one of problems in the theory of
partial differential equations.

%
% --------------------------------------------------
%      << \S 2. Main results >>
% --------------------------------------------------
%

\section{Main results}\label{section2}
%{\bf \S 2. Main results}
%

   By expanding $F(t,x,z_1,z_2)$ into Taylor series in $(z_1,z_2)$
we can express (\ref{1.2}) in the form
\begin{equation}\label{2.1}
   t \,\frac{\partial u}{\partial t}
   = a(t,x)+ \lambda(t,x)u+ b(t,x) \frac{\partial u}{\partial x}
        + R \Bigl( t,x,u, \frac{\partial u}{\partial x} \Bigr)
\end{equation}
where $R(t,x,z_1,z_2)$ is a function of the form
\[
    R(t,x,z_1,z_2)= \sum_{j+\alpha \geq 2}a_{j,\alpha}(t,x)
         {z_1}^j{z_2}^{\alpha}.
\]
\par
   Let $\mu(t) \in C^0((0,T_0])$ be a weight function
on $(0,T_0]$: this means that $\mu(t)$ is an increasing, continuous 
and positive value function on $(0,T_0]$ satisfying
\[
     \int_0^{T_0} \frac{\mu(\tau)}{\tau} d\tau < \infty.
\]
By this condition, we have $\mu(t) \longrightarrow 0$ 
(as $t \longrightarrow +0$).
\par
   In addition to ${\rm A}_1) \sim {\rm A}_3)$, we suppose $0<T_0<1$ 
and the following conditions:
\begin{align*}
   &\mbox{c1) \enskip $|a(t,x)| \leq A t^{\mu}$ on 
      $(0,T_0] \times D_{R_0}$ for some $A>0$ and $\mu>0$}. \\
   &\mbox{c2) \enskip 
           $|\lambda(t,x)-\lambda(0,x)| \leq \Lambda \mu(t)$ 
          on $(0,T_0] \times D_{R_0}$ for some $\Lambda>0$}. \\
   &\mbox{c3) \enskip $|b(t,x)| \leq \dfrac{B \mu(t)}{|\log t|}$ 
          on $(0,T_0] \times D_{R_0}$ for some $B>0$}.
\end{align*}
\par
   For simplicity, for $W=[0,T] \times D_R$ or 
$W=(0,T] \times D_R$ we denote by $X_0(W)$ the set of all functions
$f(t,x) \in C^0(W)$ which are holomorphic in $x$ for any fixed $t$.
Then, we have 
$a(t,x), \lambda(t,x), b(t,x) \in X_0([0,T_0] \times D_{R_0})$, 
and also $a_{j,\alpha}(t,x) \in X_0([0,T_0] \times D_{R_0})$.
\par
   Similarly, we denote by $X_1(W)$ the set of all functions
$f(t,x) \in C^1(W \cap \{0<t<T\})$ satisfying $f(t,x) \in X_0(W)$
and $(t \partial f/\partial t)(t,x) \in X_0(W)$.
\par
   Let us define classes of functions in which we consider
our equation (\ref{2.1}).

\begin{defin}\label{Definition2.1}
    (1) For $d>0$, $T>0$ and $R>0$ we denote by 
$X_1^d([0,T] \times D_R)$ the set of all $u(t,x) \in 
X_1([0,T] \times D_R)$ satusfying
\[
    \max\{|u(t,x)|, |(\partial u/\partial x)(t,x)|\} \leq Ct^d
        \quad \mbox{on $[0,T] \times D_R$}
\]
for some $C>0$. 
\par
   (2)  We set
\[
     X_1^d = \varinjlim_{(T,R) \to (0,0)}
           X_1^d([0,T] \times D_R)
\]
(where $\varinjlim$ denotes the inductive limit).
In other words, $X_1^d$ is the set of all germe of functions 
$u(t,x)$ belonging to $X_1^d([0,T] \times D_R)$ for some $T>0$ and $R>0$.
By the definition, if two functions 
$u_i(t,x) \in X_1^d([0,T_i] \times D_{R_i})$ ($i=1,2$) satisfy
$u_1(t,x)=u_2(t,x)$ on $[0,\epsilon] \times D_{\delta}$ for 
some $\epsilon >0$ and $\delta>0$, we think that $u_1=u_2$ in $X_1^d$.
\par
   (3) We set also
$$
    X_1^{+} = \bigcup_{d>0} X_1^d, \quad
        \mbox{and} \quad
    X_1^{(d)} = \bigcap_{0<a<d} X_1^a \quad \mbox{(for $d>0$)}.
$$
\end{defin}

   We note that $t^{\mu} \in X_1^{\mu}$, 
$t^{\mu}|\log t| \not\in X_1^{\mu}$, but 
$t^{\mu}|\log t| \in X_1^{(\mu)}$. Let us define another class of
functions.

\begin{defin}\label{Definition2.2}
    For $T>0$ and $R>0$ we denote by 
${\mathscr X}_1([0,T] \times D_R)$ the set of all $u(t,x) \in 
X_1([0,T] \times D_R)$ satisfying
\begin{equation}\label{2.2}
     \varlimsup_{\rho \to +0}  \biggl[ \, 
          \lim_{\sigma \to +0}  \Bigl( \frac{1}{\rho^2}
       \sup_{(0,\sigma) \times D_{\rho}}|u(t,x)| \Bigr)  \biggr]
        =  0.
\end{equation}
We set
\[
     {\mathscr X}_1 = \varinjlim_{(T,R) \to (0,0)}
           {\mathscr X}_1([0,T] \times D_R).
\]
\end{defin}

   It is clear that $X_1^+ \subset {\mathscr X}_1$ holds. 
If $u(t,x) \in X_1([0,T] \times D_R)$ satisfies
\[
     \lim_{t \to +0} \Bigl( \sup_{x \in D_R}|u(t,x)| \Bigr)
         = 0
\]
we have $u(t,x) \in {\mathscr X}_1$. If $u(t,x)=x^m$ for some
$m \in \{3,4,\ldots\}$ we have $u(t,x) \in {\mathscr X}_1$.
Therefore, the class ${\mathscr X}_1$ is much wider than $X_1^+$.
\par
   Now, let us state our main results of this paper.
We denote by ${\mathcal O}_0$ the set of all germs of holomorphic
functions $\psi(x)$ at $x=0$. For $X=X_1^{(\mu)}$,
$X^+_1$ or ${\mathscr X}_1$, we denote by ${\mathcal S}((\ref{2.1}),X)$ 
the set of all solutions of (\ref{2.1}) belonging to the class $X$.

\begin{thm}\label{Theorem2.3}
    Suppose the conditions ${\rm A}_1) \sim {\rm A}_3)$, 
{\rm (\ref{1.3})} and {\rm c1)} $\sim$ {\rm c3)}.
Then, we have the following results.
\par
   {\rm (1)} We have ${\mathcal S}((\ref{2.1}),X_1^{(\mu)}) \ne \emptyset$.
\par
   {\rm (2)} Take any $u_0(t,x) \in {\mathcal S}((\ref{2.1}),X_1^{(\mu)})$ 
and fix it. Then, for any $\psi(x) \in {\mathcal O}_0$ the equation 
{\rm (\ref{2.1})} has a unique solution $u(t,x) \in X^+_1$ which is 
expressed in the form 
\[
       u(t,x)=u_0(t,x)+w(t,x)
\]
for some $w(t,x) \in X^+_1$ satisfying
\[
     \lim_{t \to +0} \bigl(t^{-\lambda(0,x)}w(t,x) \bigr)
        = \psi(x) \quad \mbox{uniformly near $x=0$}.
\]
From now, we write this solution as $U(\psi)(t,x)$.
\par
   {\rm (3)} If $u_0(t,x)$ is fixed, we have
\[
    {\mathcal S}((\ref{2.1}),{\mathscr X}_1)
     ={\mathcal S}((\ref{2.1}),X_1^+)
     = \{U(\psi) \,;\, \psi(x) \in {\mathcal O}_0 \}.
\]
\end{thm}

\begin{rem}\label{Remark2.4}
    In \cite{LRT} and \cite{unique}, in the case 
${\rm Re}\lambda(0,0)<0$ we have discussed (\ref{2.1}) under the 
condition
\begin{equation}\label{2.3}
    |a(t,x)| \leq A\mu(t) \quad \mbox{on $[0,T_0] \times D_{R_0}$}.
\end{equation}
But, in the case ${\rm Re}\lambda(0,0)>0$ it is difficult to 
solve (\ref{2.1}) under (\ref{2.3}). This is the reason why we suppose 
a stronger condition $|a(t,x)| \leq At^{\mu}$ instead of (\ref{2.3}).
\end{rem}

\begin{exam}\label{Example2.5}
    In order to understand the situation in 
Theorem \ref{Theorem2.3}, let us consider the following equation:
\begin{equation}\label{2.4}
    t \frac{\partial u}{\partial t}
    = x t^{\mu} +  \lambda u 
         + \Bigl( \frac{\partial u}{\partial x} \Bigr)^2.
\end{equation}
In this case we have $a(t,x)=x t^{\mu}$ with $\mu>0$, 
$\lambda(t,x)=\lambda$ with $\lambda>0$, 
and $b(t,x)=0$.
\par
   (1) By a calculation we see that (\ref{2.4}) has a solution
$u_0(t,x)$ of the form: in the case $\lambda \ne \mu, 2\mu$
\[
     u_0(t,x)= \frac{xt^{\mu}}{\mu-\lambda}
           + \frac{t^{2\mu}}{(\mu-\lambda)^2(2\mu-\lambda)};
\]
in the case $\lambda=2\mu$
\[
     u_0(t,x)= \frac{xt^{\mu}}{\mu-\lambda}
           + \frac{t^{2\mu}(\log t)}{(\mu-\lambda)^2};
\]
in the case $\lambda=\mu$
\[
     u_0(t,x)= xt^{\mu}(\log t)
           + \frac{t^{2\mu}(\log t)^2}{\mu}
        - \frac{2t^{2\mu}(\log t)}{\mu^2}
        + \frac{2t^{2\mu}}{\mu^3}.
\]
Thus, in the case $\lambda \ne \mu$ we have a solution 
$u_0(t,x) \in X_1^{\mu} \subset X_1^{(\mu)}$,
and in the case $\lambda=\mu$ we have a solution 
$u_0(t,x) \in X^{(\mu)}_1$ (but, $u_0(t,x) \not\in X^{\mu}_1$).
\par
   (2) Next, let us look for a solution of the form
$u(t,x)=u_0(t,x)+w(t,x)$. Our equation (\ref{2.4}) is reduced to 
an equation with respect to $w$:
\begin{equation}\label{2.5}
   t \frac{\partial w}{\partial t}
    = \lambda w + f(t) \frac{\partial w}{\partial x}
         + \Bigl( \frac{\partial w}{\partial x} \Bigr)^2
\end{equation}
with
\[
     f(t)= \left\{ \begin{array}{ll}
        2t^{\mu}/(\mu-\lambda),
          &\mbox{if $\lambda \ne \mu$}, \\[3pt]
         2t^{\mu}(\log t),
          &\mbox{if $\lambda= \mu$}.
     \end{array}  \right.
\]
Then, for any $\psi(x) \in {\mathcal O}_0$ the equation (\ref{2.5})
has a solution $W(\psi)$ of the form
\[
    W(\psi)= \sum_{k \geq 0}\psi_k(\eta(t)+x) t^{(k+1)\lambda}
         \quad
          \mbox{with $\psi_0(x)=\psi(x)$}
\]
where 
\[
    \eta(t)= \int_0^t \frac{f(\tau)}{\tau} d\tau
    = \left\{ \begin{array}{ll}
        2t^{\mu}/(\mu(\mu-\lambda)),
          &\mbox{if $\lambda \ne \mu$}, \\[3pt]
         2t^{\mu}(\log t)/\mu-2t^{\mu}/\mu^2,
          &\mbox{if $\lambda= \mu$}, 
     \end{array}  \right.
\]
and $\psi_k(x)$ ($k \geq 1$) are determined 
by the following recurrence formula:
\[
     \psi_k(x)= \frac{1}{k\lambda}
           \sum_{i+j=k-1}\psi_i^{(1)}(x)
                  \psi_j^{(1)}(x).
\]
\par
   (3) By Theorem \ref{Theorem2.3} we have
\[
    {\mathcal S}((\ref{2.4}),{\mathscr X}_1)
     ={\mathcal S}((\ref{2.4}),X_1^+)
     = \{u_0 + W(\psi) \,;\, \psi(x) \in {\mathcal O}_0 \}.
\]
\end{exam}

\begin{exam}\label{Example2.6}
      The following example will help the readers' 
understanding of the condition (\ref{2.2}). The equation
\begin{equation}\label{2.6}
    t \frac{\partial u}{\partial t}= u 
         - \Bigl( \frac{\partial u}{\partial x} \Bigr)^2
\end{equation}
has a solution $u=x^2/4$, and in this case we have 
\[
     \varlimsup_{\rho \to +0}  \biggl[ \, 
          \lim_{\sigma \to +0}  \Bigl( \frac{1}{\rho^2}
       \sup_{(0,\sigma) \times D_{\rho}}|u(t,x)| \Bigr)  \biggr]
        =  \frac{1}{4}
\]
and so $x^2/4 \not\in {\mathscr X}_1$. 
If we take a function class ${\mathscr F}$ which contains $x^2/4$,
then we have ${\mathcal S}((\ref{2.6}),{\mathscr F}) 
                            \ne {\mathcal S}((\ref{2.6}),X_1^+)$.
\par
   Thus, (\ref{2.2}) will be an optimal condition for 
${\mathcal S}((\ref{2.1}),{\mathscr X}_1)
              ={\mathcal S}((\ref{2.1}),X_1^+)$ to be valid.
\end{exam}

   As is seen in Example \ref{Example2.6}, some equations have
solutions which do not belong to ${\mathscr X}_1$. At present, the
author has no idea to treat such solutions.
\par
   The rest part of this paper is organized as follows.
In the next \S 3, we show 
${\mathcal S}((\ref{2.1}),X_1^{(\mu)}) \ne \emptyset$
in the case $\mu>{\rm Re}\lambda(0,0)$, as an application of
results in \cite{LRT} and \cite {unique}. In \S 4, we show 
${\mathcal S}((\ref{2.1}),X_1^{(\mu)}) \ne \emptyset$
in the case $\mu \leq {\rm Re}\lambda(0,0)$. In \S 5, we reduce
our problem to the case $a(t,x) \equiv 0$, and in the last 
\S 6, we show (2) and (3) of Theorem \ref{Theorem2.3}. in the case 
$a(t,x) \equiv 0$.

%
% --------------------------------------------------
% << \S 3. Solvability in the case 
%                $\mu>{\rm Re}\lambda(0,0)$ >>
% --------------------------------------------------
%

\section{Solvability in the case $\mu>{\rm Re}\lambda(0,0)$}
                \label{section3}
%{\bf \S 3. Solvability in the case 
%                $\mu>{\rm Re}\lambda(0,0)$}
%

   In this section, we consider the equation (\ref{2.1}) under the 
following conditions:
\begin{align}
    &|a(t,x)| \leq A t^{\mu} 
              \quad \mbox{on $[0,T_0] \times D_{R_0}$}, 
         \label{3.1}\\
    &|b(t,x)| \leq B \mu(t) 
              \quad \mbox{on $[0,T_0] \times D_{R_0}$}, 
        \label{3.2}
\end{align}
where $A>0$, $\mu>0$, $B>0$ and $\mu(t)$ is a weight function
on $(0,T_0]$.

\begin{thm}\label{Theorem3.1} 
   Suppose {\rm (\ref{3.1})}, {\rm (\ref{3.2})} and 
$0< {\rm Re}\lambda(0,0) < \mu$. Then, the equation {\rm (\ref{2.1})} 
has a unique solution $u(t,x) \in X_1^{(\mu)}$.
\end{thm}

   Since this is an easy consequence of results in 
\cite{LRT, unique}, let us recall some results in \cite{LRT, unique}.
For a weight function $\mu(t)$ we define a function $\varphi(t)$
by
\begin{equation}\label{3.3}
     \varphi(t)= \int_0^t \frac{\mu(\tau)}{\tau} d\tau, 
     \quad 0<t \leq T_0,
\end{equation}
and for $T>0$, $R>0$ and $r>0$ we set
\[
    W_{T,R,r}=\{(t,x) \in [0,T] \times D_R \,;\,
          \varphi(t)/r+|x|<R \}.
\]

\begin{thm}\label{Theorem3.2}
    Suppose the conditions {\rm (\ref{3.2})}, 
${\rm Re}\lambda(0,0)<0$, and
\[
       |a(t,x)| \leq A\mu(t) \quad
                  \mbox{on $[0,T_0] \times D_{R_0}$}
\]
for some $A>0$. Then, we have the following results.
\par
   {\rm (1)(Theorem 1.1 (with $\alpha=1$) in \cite{LRT})}. 
There are $T>0$, $R>0$, $r>0$ 
and $M>0$ such that equation 
{\rm (\ref{2.1})} has a unique solution $u_0(t,x) \in X_1(W_{T,R,r})$ 
satisfying
\[
    |u_0(t,x)| \leq M\mu(t) \enskip \mbox{and} \enskip
    \Bigl| \frac{\partial u_0}{\partial x}(t,x) \Bigr| 
       \leq M\mu(t) \quad \mbox{on $W_{T,R,r}$}.
\]
\par
   {\rm (2)(Theorem 2.2 in \cite{unique})}.
Let $u(t,x) \in X_1((0,T] \times D_R)$ be a solution of 
{\rm (\ref{2.1})} with $T>0$ and $R>0$. If $u(t,x)$ satisfies
\[
     \varlimsup_{\rho \to +0} \, \biggl[ 
          \lim_{\sigma \to +0}  \Bigl( \frac{1}{\rho^2}
       \sup_{(0,\sigma) \times D_{\rho}}|u(t,x)| \Bigr)  \biggr]
        =  0, 
\]
we have $u(t,x)=u_0(t,x)$ on $(0,\epsilon] \times D_{\delta}$ for 
some $\epsilon>0$ and $\delta>0$, where $u_0(t,x)$ is the solution 
of {\rm (\ref{2.1})} obtained in {\rm (1)}.
\end{thm}

\begin{proof}[Proof of Theorem \ref{Theorem3.1}]
   Suppose (3.1), (3.2) and $0< {\rm Re}\lambda(0,0) < \mu$. 
Take any $d>0$ satisfying ${\rm Re}\lambda(0,0) < d<\mu$, and set
$u=t^dw$. Then our equation (\ref{2.1}) is written as an equation 
with respect to $w$:
\begin{align}
   t \,\frac{\partial w}{\partial t}
   = t^{-d}a(t,x) &+ (\lambda(t,x)-d)w+ b(t,x) \frac{\partial w}{\partial x} 
             \label{3.4}\\
   &\qquad+ \sum_{j+\alpha \geq 2} a_{j,\alpha}(t,x) t^{d(j+\alpha-1)}
           w^j\Bigl(\frac{\partial w}{\partial x} \Bigr)^{\alpha}.
        \notag
\end{align}
If we set $\mu_0(t)=\mu(t)+t^{\mu-d}$, this $\mu_0(t)$ is also
a weight function on $(0,T_0]$ and we have
\begin{align*}
   &|t^{-d}a(t,x)| \leq A\mu_0(t) 
              \quad \mbox{on $[0,T_0] \times D_{R_0}$}, \\
   &|b(t,x)| \leq B\mu_0(t) 
              \quad \mbox{on $[0,T_0] \times D_{R_0}$}.
\end{align*}
Since ${\rm Re}(\lambda(0,0)-d)<0$ holds, by (1) of 
Theorem \ref{Theorem3.2} we have a solution 
$w(t,x) \in X_1([0,T] \times D_R)$ of (\ref{3.4}) satisfying
\[
    |w(t,x)| \leq M\mu_0(t) \enskip \mbox{and} \enskip
    \Bigl| \frac{\partial w}{\partial x}(t,x) \Bigr| 
       \leq M\mu_0(t) \quad \mbox{on $[0,T] \times D_R$}
\]
for some $T>0$, $R>0$ and $M>0$. Thus, by setting
$u(t,x)=t^dw(t,x)$ we have a solution $u(t,x) \in X^d_1$
of (\ref{2.1}) satisfying
\begin{equation}\label{3.5}
     \lim_{t \to +0} \Bigl( \sup_{x \in D_R}|t^{-d}u(t,x)| \Bigr)
      = 0 
\end{equation}
for some $R>0$. 
\par
   The uniqueness of such a solution is verified as follows.
Let $u(t,x) \in X^d_1$ be a solution of (\ref{2.1}) satisfying
(\ref{3.5}). Set $w(t,x)=t^{-d}u(t,x)$: then $w(t,x)$ is a solution
of (\ref{3.4}). By the condition (\ref{3.5}) we can apply (2) of 
Theorem \ref{Theorem3.2} to this situation. Hence we have the 
uniqueness of $w(t,x)$ which yields the uniqueness of $u(t,x)$.
\par
   Since $d$ is taken arbitrarily so that
${\rm Re}\lambda(0,0) < d<\mu$, we have the result 
$u(t,x) \in X_1^{(\mu)}$.
\end{proof}

%
% --------------------------------------------------
%    << \S 4. Solvability in the case 
%             $0<\mu \leq {\rm Re}\lambda(0,0)$ >>
% --------------------------------------------------
%

\section{Solvability in the case $0<\mu \leq {\rm Re}\lambda(0,0)$}
       \label{section3}
%{\bf \S 4. Solvability in the case 
%                 $0<\mu \leq {\rm Re}\lambda(0,0)$}
%

    In this section, we will show

\begin{thm}\label{Theorem4.1}
    Suppose {\rm (\ref{3.1})}, {\rm (\ref{3.2})} and 
$0<\mu \leq {\rm Re}\lambda(0,0)$. Then, the equation {\rm (\ref{2.1})} 
has a solution $u(t,x) \in X_1^{\mu}$ in the case 
$\mu< {\rm Re}\lambda(0,0)$ {\rm (}resp. $u(t,x) \in X_1^{(\mu)}$ 
in the case $\mu={\rm Re}\lambda(0,0)${\rm )}.
\end{thm}

   Set
\[
    L= t \frac{\partial}{\partial t}-\lambda(t,x)
        -b(t,x) \frac{\partial}{\partial x}.
\]
In the next subsection 4.1 we solve the Cauchy problem
\begin{equation}\label{4.1}
    Lw=g(t,x), \quad w \bigr|_{t=T}=\psi(x). 
\end{equation}
Then, by using the solvability of this Cauchy problem and 
Theorem \ref{Theorem3.1}, in subsection 4.2 we give a proof of 
Theorem \ref{Theorem4.1}.

%
% --------------------------------------------------
%  << 4.1. On the Cauchy problem (4.1) >>
% --------------------------------------------------
%

\subsection{On the Cauchy problem (\ref{4.1})}\label{subsection4.1}
%{\bf 4.1. On the Cauchy problem (4.1)}
%

   Let $\varphi(t)$ be as in (\ref{3.3}). Take $0<T<T_0$ and 
$0<R <R_1<R_0$ so that
\[
    R+B\varphi(T)< R_1, \quad R_1+B\varphi(T) <R_0.
\]
For $(t_0,x_0) \in [0,T] \times D_{R_1}$, we consider the 
initial value problem
\begin{equation}\label{4.2}
    \frac{dx}{dt}= - \frac{b(t,x)}{t}, 
    \quad x \bigr|_{t=t_0}=x_0. 
\end{equation}
Since (\ref{3.2}) is supposed, it is easy to see that (\ref{4.2}) 
has a unique solution $x(t)$ on $[0,T]$: we write this solution
as $\phi(t;t_0,x_0)$. Set 
\[
    {\mathcal D}= \bigcup_{x_0 \in D_{R_1}}
      \{(t, \phi(t;T,x_0)) \,;\, 0<t \leq T \}.
\]
It is also easy to see that 
\[
     (0,T] \times D_R \subset {\mathcal D}
           \subset (0,T] \times D_{R_0}
\]
and that ${\mathcal D} \cap \{t=t_0 \}$ is an open
subset of $\BC$ for any $t_0 \in (0,T]$.

\begin{prop}\label{Proposition4.2}
    Under the above situation, for any $g(t,x) \in X_0({\mathcal D})$ 
and any holomorphic function $\psi(x)$ on $D_{R_1}$ the Cauchy problem 
{\rm (\ref{4.1})} has a unique solution $w(t,x) \in X_1({\mathcal D})$,
and it is given by
\begin{align}
   w(t,x)= &\psi(\phi(T;t,x)) \exp \Bigl[ -\int_{t}^{T}
     \frac{\lambda(s,\phi(s;t,x))}{s}ds \Bigr]
       \label{4.3}\\
       &- \int_t^T \exp \Bigl[ -\int_t^{\tau}
     \frac{\lambda(s,\phi(s;t,x))}{s}ds \Bigr]
        \frac{g(\tau,\phi(\tau;t,x))}{\tau} d\tau. \notag
\end{align}
\end{prop}

\begin{proof}
    Let $w(t,x) \in X_1({\mathcal D})$ be a 
solution of (\ref{4.1}). Take any $x_0 \in D_{R_1}$ and set
\begin{align*}
   &w^*(t)=w(t,\phi(t;T,x_0)), \\
   &\lambda^*(t)= \lambda(t,\phi(t;T,x_0)), \\
   &g^*(t)=g(t,\phi(t;T,x_0)).
\end{align*}
Then, our problem (\ref{4.1}) is written in the form
\[
     t \frac{dw^*}{dt} - \lambda^*(t)w^*=g^*(t),
     \quad w^*(T)=\psi(x_0).
\]
By integrating this we have
\begin{align*}
   w^*(t)= &\psi(x_0) \exp \Bigl[ -\int_{t}^{T}
     \frac{\lambda^*(s)}{s}ds \Bigr]\\
       &\qquad \qquad - \int_t^T \exp \Bigl[ -\int_t^{\tau}
     \frac{\lambda^*(s)}{s}ds \Bigr] \frac{g^*(\tau)}{\tau}
        d\tau.
\end{align*}
Therefore, we have
\begin{align*}
   w(t,\phi(t;T,x_0))= &\psi(x_0) \exp \Bigl[ -\int_{t}^{T}
     \frac{\lambda(s,\phi(s;T,x_0))}{s}ds \Bigr] \\
       &- \int_t^T \exp \Bigl[ -\int_t^{\tau}
     \frac{\lambda(s,\phi(s;T,x_0))}{s}ds \Bigr] 
       \frac{g(\tau,\phi(\tau;T,x_0))}{\tau} d\tau.
\end{align*}
\par
   Take any $(t,x) \in {\mathcal D}$; then we have a unique 
$x_0 \in D_{R_1}$ such that $x=\phi(t;T,x_0)$.
In this case, we have $x_0=\phi(T;t,x)$ and 
$\phi(s;T,x_0)=\phi(s;t,x)$ for $0 <s \leq T$. Therefore,
by replacing $x_0$ and $\phi(s;T,x_0)$ by $\phi(T;t,x)$ and 
$\phi(s;t,x)$, respectively,  in the above formula we finally 
obtain
\begin{align*}
   w(t,x)= &\psi(\phi(T;t,x)) \exp \Bigl[ -\int_{t}^{T}
     \frac{\lambda(s,\phi(s;t,x))}{s}ds \Bigr] \\
       &- \int_t^T \exp \Bigl[ -\int_t^{\tau}
     \frac{\lambda(s,\phi(s;t,x))}{s}ds \Bigr]
        \frac{g(\tau,\phi(\tau;t,x))}{\tau} d\tau.
\end{align*}
This proves (\ref{4.3}) and also the uniqueness
of the solution of (\ref{4.1}).
\par
   To show the existence of a solution, we have only to notice
that the right side of (\ref{4.3}) is well-defined as a function 
in $X_1({\mathcal D})$. Thus, the existence of a solution is also 
shown. 
\end{proof}

\begin{cor}\label{Corollary4.3}
    Suppose that 
$0<a_0 \leq {\rm Re}\lambda(t,x)$ holds on ${\mathcal D}$.
If $\psi(x)$ and $(\partial \psi/\partial x)(x)$ are bounded 
on $D_{R_1}$ and if $g(t,x) \in X_1^{\mu}({\mathcal D})$ for some 
$\mu$ satisfying $0<\mu<a_0$, then the unique solution $w(t,x)$ 
of {\rm (\ref{4.1})} belongs to $X_1^{\mu}({\mathcal D})$.
\end{cor}

\begin{proof}
    Suppose that $|\psi(x)| \leq M$ on $D_{R_1}$ for
some $M>0$, and $|g(t,x)| \leq Gt^{\mu}$ on ${\mathcal D}$ for some
$G>0$. Since
\[
    \Bigl|\exp \Bigl[ -\int_{t}^{T}
         \frac{\lambda(s,\phi(s;T,x_0))}{s}ds \Bigr] \Bigr|
    \leq \exp \Bigl[ -\int_{t}^{T}
     \frac{a_0}{s}ds \Bigr]
      = \Bigl( \frac{t}{T} \Bigr)^{a_0}
\]
holds, by (\ref{4.3}) we have
\begin{align*}
   |w(t,x)| &\leq M\Bigl( \frac{t}{T} \Bigr)^{a_0}
       + \int_t^T \Bigl( \frac{t}{\tau} \Bigr)^{a_0}
        \frac{G\tau^{\mu}}{\tau} d\tau \\
    &= M\Bigl( \frac{t}{T} \Bigr)^{a_0}
      + Gt^{a_0} \Bigl( \frac{-1}{(a_0-\mu)T^{a_0-\mu}}
          + \frac{1}{(a_0-\mu)t^{a_0-\mu}}\Bigr) \\
    &\leq M\Bigl( \frac{t}{T} \Bigr)^{a_0}
          + Gt^{a_0} \frac{1}{(a_0-\mu)t^{a_0-\mu}}
     = M\Bigl( \frac{t}{T} \Bigr)^{a_0}
            + \frac{G}{(a_0-\mu)} t^{\mu}
\end{align*}
on ${\mathcal D}$. We can get the estimate of 
$|(\partial w/\partial x)(t,x)|$ in the same way. 
\end{proof}

%
% --------------------------------------------------
%  << 4.2. Proof of Theorem 4.1 >>
% --------------------------------------------------
%

\subsection{Proof of Theorem 4.1}\label{subsection4.2}
%{\bf 4.2. Proof of Theorem 4.1}
%

   Suppose (\ref{3.1}), (\ref{3.2}) and 
$0<\mu \leq {\rm Re}\lambda(0,0)$. We set 
$N=[{\rm Re}\lambda(0,0)/\mu]$, the integer part of 
${\rm Re}\lambda(0,0)/\mu$; then we have
$N\mu \leq {\rm Re}\lambda(0,0)<(N+1)\mu$.
\par
    If $N\mu<{\rm Re}\lambda(0,0)$, we set $d=\mu$. If 
$N\mu = {\rm Re}\lambda(0,0)$, we take $0<d<\mu$ so that
$Nd<{\rm Re}\lambda(0,0)<(N+1)d$. Then, in any case we have
$$
    0< d \leq \mu < 2d \quad \mbox{and} \quad
     Nd< {\rm Re}\lambda(0,0) <(N+1)d.
$$
After that, we take $0<T^*<T_0$, $0<R^*<R_0$ and $a_0<a_1$ so that
\[
    Nd <a_0 \leq {\rm Re}\lambda(t,x) \leq a_1 < (N+1)d 
   \quad \mbox{on $[0,T^*] \times D_{R^*}$}.
\]
\par
   Under this situation, let us look for a solution of (\ref{2.1})
of the form
\[
    u(t,x)= v(t,x)+w(t,x) \quad \mbox{with}
        \quad v(t,x)=\sum_{n=1}^N u_n(t,x).
\]
Actually, we define $u_n(t,x)$ ($n=1,\ldots,N$) and 
$w(t,x)$ by the following recurrence formulas:
\begin{equation}\label{4.4}
     Lu_1=a(t,x), 
\end{equation}
for $2 \leq n \leq N$
\begin{equation}\label{4.5}
    Lu_n = \sum_{2 \leq j+\alpha \leq n} \!\! a_{j,\alpha}(t,x)
           \sum_{|\vec{k}|+|\vec{m}|=n}
           u_{k_1}\cdots u_{k_j} 
           \frac{\partial u_{m_1}}{\partial x} \cdots
         \frac{\partial u_{m_{\alpha}}}{\partial x}
\end{equation}
(where $\vec{k}=(k_1,\ldots,k_j)$, $|\vec{k}|=k_1+\cdots+k_j$,
$\vec{m}=(m_1,\ldots,m_{\alpha})$ and 
$|\vec{m}|=m_1+\cdots+m_{\alpha}$), and
\begin{align}
    Lw = &R \Bigl(t,x, v+w, \frac{\partial v}{\partial x}
            + \frac{\partial w}{\partial x} \Bigr)
         - R \Bigl(t,x, v, \frac{\partial v}{\partial x} 
                  \Bigr) \label{4.6}\\
   &+ \sum_{j+\alpha \geq 2} \!\! a_{j,\alpha}(t,x)
           \sum_{|\vec{k}|+|\vec{m}| \geq N+1}
           u_{k_1}\cdots u_{k_j} 
           \frac{\partial u_{m_1}}{\partial x} \cdots
         \frac{\partial u_{m_{\alpha}}}{\partial x}. \notag
\end{align}

    Step 1: By applying Proposition \ref{Proposition4.2} 
(and Corollary \ref{Corollary4.3}) to (\ref{4.4}) we have a solution 
$u_1(t,x) \in X_1^{d}([0,T_1] \times D_{R_1})$ for some $0<T_1<T^*$ 
and $0<R_1<R^*$. Then, the right side of (\ref{4.5}) (with $n=2$) 
belongs to $X_1^{2d}([0,T_1] \times D_{R_1})$ and so in the case 
$2 \leq N$ we have a solution 
$u_2(t,x) \in X_1^{2d}([0,T_2] \times D_{R_2})$ for some $0<T_2<T_1$ 
and $0<R_2<R_1$. 
\par
   By repeating the same argument we have 
$0<T_N< \cdots<T_2<T_1$, $0<R_N< \cdots<R_2<R_1$ and
$u_n(t,x) \in X_1^{nd}([0,T_n] \times D_{R_n})$ ($1 \leq n \leq N$)
so that they satisfy the recurrence formulas (\ref{4.4}) and 
(\ref{4.5}) ($2 \leq n \leq N$).

    Step 2:  Now, let us solve (\ref{4.6}). We set
\[
   f(t,x)=\sum_{j+\alpha \geq 2} \!\! a_{j,\alpha}(t,x)
           \sum_{|\vec{k}|+|\vec{m}| \geq N+1}
           u_{k_1}\cdots u_{k_j} 
           \frac{\partial u_{m_1}}{\partial x} \cdots
         \frac{\partial u_{m_{\alpha}}}{\partial x}:
\]
we have the conditions
$f(t,x) \in X_0([0,T_N] \times D_{R_N})$ and 
\begin{equation}\label{4.7}
       |f(t,x)| \leq Ft^{(N+1)d} \quad  
       \mbox{on $[0,T_N] \times D_{R_N}$}
\end{equation}
for some $F>0$. Since $v(t,x)$ is a known function in 
$X_1^d([0,T_N] \times D_{R_N})$, we have the expression
\begin{align*}
   &R \Bigl(t,x, v+w, \frac{\partial v}{\partial x}
            + \frac{\partial w}{\partial x} \Bigr)
         - R \Bigl(t,x, v, \frac{\partial v}{\partial x} 
                 \Bigr) \\
   &=c_1(t,x)w + c_2(t,x)\frac{\partial w}{\partial x}
     + \sum_{j+\alpha \geq 2} c_{j,\alpha}(t,x) w^j
     \Bigl(\frac{\partial w}{\partial x} \Bigr)^{\alpha},
\end{align*}
where $c_i(t,x) \in X_0([0,T_N] \times D_{R_N})$
($i=1,2$), $c_{j,\alpha}(t,x) \in X_0([0,T_N] \times D_{R_N})$
($j+\alpha \geq 2$), and we have
\begin{equation}\label{4.8}
     |c_i(t,x)| \leq C_it^d \quad 
       \mbox{on $[0,T_N] \times D_{R_N}$}
\end{equation}
for some $C_i>0$ ($i=1,2$).
\par
   Hence, under the setting
\begin{align*}
   &\lambda_c(t,x)= \lambda(t,x)+c_1(t,x), \\
   &b_c(t,x)= b(t,x)+c_2(t,x), \\
   &L_c= t \frac{\partial}{\partial t}-\lambda_c(t,x) 
          -b_c(t,x)\frac{\partial}{\partial x}
\end{align*}
our equation (\ref{4.6}) is written in the form
\begin{equation}\label{4.9}
    L_cw = f(t,x)+ \sum_{j+\alpha \geq 2} c_{j,\alpha}(t,x) w^j
     \Bigl(\frac{\partial w}{\partial x} \Bigr)^{\alpha}.
\end{equation}
Since (\ref{4.8}) is valid, we have
\begin{align*}
   &{\rm Re}\lambda_c(0,0)={\rm Re}\lambda(0,0) < a_1, \\
   &|b_c(t,x)| \leq (B+C_2)(\mu(t)+t^d) \quad 
          \mbox{on $[0,T_N] \times D_{R_N}$}.
\end{align*}
Since (\ref{4.7}) is satisfied and since $(N+1)d> a_1$ holds, we 
can apply Theorem \ref{Theorem3.1} to (\ref{4.9}) and we obtain a 
solution $w(t,x) \in X_1^{((N+1)d)}$.

    Step 3:  Thus, by setting
\[
    u(t,x)= \sum_{n=1}^N u_n(t,x)+w(t,x)
\]
we have a solution $u(t,x)$ of (\ref{2.1}) belonging to $X_1^{d}$.
\par
   Now, let us recall the construction of $u_1(t,x)$: it is a solution 
of (\ref{4.4}). Since (\ref{3.1}) is supposed, by 
Corollary \ref{Corollary4.3} we have 
$u_1(t,x) \in X_1^{\mu}$ if $\mu<{\rm Re}\lambda(0,0)$ (resp. 
$u_1(t,x) \in X_1^{(\mu)}$ if $\mu={\rm Re}\lambda(0,0)$). Hence, we 
have $u(t,x) \in X_1^{\mu}$ if $\mu<{\rm Re}\lambda(0,0)$ (resp. 
$u(t,x) \in X_1^{(\mu)}$ if $\mu={\rm Re}\lambda(0,0)$). This 
completes the proof of Theorem \ref{Theorem4.1}.  \qed

%
% --------------------------------------------------
% << \S 5. Reduction of the problem >>
% --------------------------------------------------
%

\section{Reduction of the problem}
%{\bf \S 5. Reduction of the problem}
%
   By Theorem \ref{Theorem3.1} and Theorem \ref{Theorem4.1} 
we have the result (1) of Theorem \ref{Theorem2.3}.  Take any 
$u_0(t,x) \in {\mathcal S}((\ref{2.1}),X_1^{(\mu)})$ and fix it.
\par
   Now, let $u(t,x)$ be a solution of (\ref{2.1}). Set
\[
      U(t,x)=u(t,x)-u_0(t,x):
\]
then, equation (\ref{2.1}) is transformed to an equation
with respect to $U(t,x)$:
\begin{align}
     t \frac{\partial U}{\partial t} = &\lambda(t,x)U
          +b(t,x)\frac{\partial U}{\partial x} \label{5.1}\\
   &+ R \Bigl(t,x, u_0+U, \frac{\partial u_0}{\partial x}
            + \frac{\partial U}{\partial x} \Bigr)
         - R \Bigl(t,x, u_0, 
            \frac{\partial u_0}{\partial x}\Bigr). \notag
\end{align}
Hence, to show the results (2) and (3) of Theorem \ref{Theorem2.3} 
it is sufficient to prove the following theorem.

\begin{thm}\label{Theorem5.1}
    Suppose the conditions ${\rm A}_1) \sim {\rm A}_3)$, 
{\rm (\ref{1.3})}, and ${\rm c1)} \sim {\rm c3)}$.  Then, we have the 
following results.
\par
   {\rm (1)} For any $\psi(x) \in {\mathcal O}_0$ the equation 
{\rm (\ref{5.1})} has a unique solution $U(t,x) \in X^+_1$ 
satisfying
\[
     \lim_{t \to +0} \bigl(t^{-\lambda(0,x)}U(t,x) \bigr)
        = \psi(x) \quad \mbox{uniformly near $x=0$}.
\]
From now, we write this solution as $U(\psi)(t,x)$.
\par
   {\rm (2)} We denote by ${\mathcal S}((\ref{5.1}),X)$ 
(where $X={\mathscr X}_1$ or $X=X_1^+$) the set of all solutions of 
{\rm (\ref{5.1})} belonging to $X$.
We have
\[
    {\mathcal S}((\ref{5.1}),{\mathscr X}_1)
     ={\mathcal S}((\ref{5.1}),X_1^+)
     = \{U(\psi) \,;\, \psi(x) \in {\mathcal O}_0 \}.
\]
\end{thm}

    The rest part of this paper will be used to prove
Theorem \ref{Theorem5.1}. Take any $0<d<\mu$. Since $u_0(t,x)$ 
is a known function belonging to $X_1^d$, we have the expression
\begin{align*}
   &R \Bigl(t,x, u_0+U, \frac{\partial u_0}{\partial x}
            + \frac{\partial U}{\partial x} \Bigr)
         - R \Bigl(t,x, u_0, 
            \frac{\partial u_0}{\partial x}\Bigr) \\
   &=c_1(t,x)U + c_2(t,x)\frac{\partial U}{\partial x}
     + \sum_{j+\alpha \geq 2} c_{j,\alpha}(t,x) U^j
     \Bigl(\frac{\partial U}{\partial x} \Bigr)^{\alpha}
\end{align*}
where $c_i(t,x) \in X_0([0,T] \times D_{R})$
($i=1,2$), $c_{j,\alpha}(t,x) \in X_0([0,T] \times D_{R})$
($j+\alpha \geq 2$) for some $T>0$ and $R>0$, and we have
$|c_i(t,x)| \leq C_it^d$ on $[0,T] \times D_{R}$ for some $C_i>0$ 
($i=1,2$). Then, our equation (\ref{5.1}) is written in the form
\begin{equation}\label{5.2}
     t \frac{\partial U}{\partial t} = \lambda_c(t,x)U
          +b_c(t,x)\frac{\partial U}{\partial x}
       + \sum_{j+\alpha \geq 2} c_{j,\alpha}(t,x) U^j
     \Bigl(\frac{\partial U}{\partial x} \Bigr)^{\alpha},
\end{equation}
where
\begin{align*}
   &\lambda_c(t,x)= \lambda(t,x)+c_1(t,x), \\
   &b_c(t,x)= b(t,x)+c_2(t,x).
\end{align*}
Without loss of generality, we may asuume: $0<T<1/e$.
By (\ref{1.3}), c2) and c3) we have 
\begin{align*}
   &{\rm Re}\lambda_c(0,0)= {\rm Re}\lambda(0,0)>0, \\
   &|\lambda_c(t,x)-\lambda_c(0,x)|
     \leq |\lambda(t,x)-\lambda(0,x)|+|c_1(t,x)| \\
   &\hspace{40mm}\leq 
         \Lambda_1(\mu(t)+t^d) \leq \Lambda_1(\mu(t)+ |\log t|t^d), \\
   &|b_c(t,x)| \leq |b(t,x)|+|c_2(t,x)|
           \leq \frac{B_1(\mu(t)+ |\log t|t^d)}{|\log t|}
\end{align*}
on $[0,T] \times D_R$ for some $\Lambda_1>0$ and $B_1>0$.
Since $\mu_0(t)= \mu(t)+ |\log t|t^d$ is also a weight function
on $(0,T]$, the equation (\ref{5.2}) satisfies
similar conditions as (\ref{1.3}), c2) and c3): the only difference is 
that we have $a(t,x) \equiv 0$ in the case (\ref{5.2}).
\par
   Thus, for the sake of simplicity of notations, instead of (\ref{5.2})
we may treat (\ref{2.1}) under the condition $a(t,x) \equiv 0$.

%%%%%%%%%%%%%%%%%%%%%%%%%%%%%%%%%%%%%%%%%%%%%%%%%%%
\makeatletter
\@addtoreset{equation}{subsection} 
\def\theequation{\thesubsection.\arabic{equation}}
\makeatother
%%%%%%%%%%%%%%%%%%%%%%%%%%%%%%%%%%%%%%%%%%%%%%%%%%%

%
% --------------------------------------------------
% << \S 6. Analysis in the case $a(t,x) \equiv 0$ >>
% --------------------------------------------------
%

\section{Analysis in the case $a(t,x) \equiv 0$}\label{section6}
%{\bf \S 6. Analysis in the case $a(t,x) \equiv 0$}
%

   In this section, we consider the equation (\ref{2.1}) in
the case $a(t,x) \equiv 0$; then our equation is
\begin{equation}\label{6.0.1}
   t \,\frac{\partial u}{\partial t}
   = \lambda(t,x)u+ b(t,x) \frac{\partial u}{\partial x}
        + R \Bigl( t,x,u, \frac{\partial u}{\partial x} \Bigr)
\end{equation}
which has a trivial solution $u(t,x) \equiv 0$.

%
% --------------------------------------------------
%  << 6.1. On a basic equation $Lv=0$ >>
% --------------------------------------------------
%

\subsection{On a basic equation $Lv=0$}\label{subsection6.1}
%{\bf 6.1. On a basic equation $Lv=0$}
%

   First, let
\[
     L= t \frac{\partial }{\partial t}
    - \lambda(t,x) - b(t,x) \frac{\partial }{\partial x}
\]
be as before, and let us consider the equation
\begin{equation}\label{6.1.1}
    Lv=0  
\end{equation}
in the case $\lambda(0,x) \equiv 0$ on $D_{R_0}$, 
under the conditions: 
$\lambda(t,x), b(t,x) \in X_0([0,T_0] \times D_{R_0})$ and
\begin{align}
   &|\lambda(t,x)| \leq \Lambda \mu(t) \quad 
              \mbox{on $[0,T_0] \times D_{R_0}$}, \label{6.1.2}\\
   &|b(t,x)| \leq B \mu(t) \quad 
              \mbox{on $[0,T_0] \times D_{R_0}$} \label{6.1.3}
\end{align}
for some $\Lambda>0$, $B>0$ and a weight function $\mu(t)$
on $(0,T_0]$. 
\par
    We take $0<T<T_0$, $0<R<R_0$, $0<r \leq 1/B$ and set
\[
     W=\{(t,x) \in [0,T] \times D_R \,;\,
           \varphi(t)/r+|x|<R \}.
\]
Let $\phi(t;t_0,x_0)$ be the unique solution of (\ref{4.2}) as before.
For $x_0 \in D_R$ we set
\[
    t_{x_0}= \sup\{t_1>0 \,;\, (t,\phi(t;0,x_0)) \in W \enskip
        \mbox{for any $0 \leq t<t_1$} \}.
\]

\begin{llem}\label{Lemma6.1.1}
   {\rm (1)} For any $(t_0,x_0) \in W$, 
we have $(t,\phi(t;t_0,x_0)) \in W$ for any $0 \leq t \leq t_0$.
\par
   {\rm (2)} We have
\[
    W = \bigcup_{x_0 \in D_R}
      \{(t, \phi(t;0,x_0)) \,;\, 0 \leq t<t_{x_0} \}.
\]
\end{llem}

\begin{proof}
   Set $x(t)=\phi(t;t_0,x_0)$: then we have
\[
     x(t)= x_0 + \int_t^{t_0} \frac{b(\tau,x(\tau))}{\tau}
         d\tau.
\]
If $x(t)$ exists on $(t_1,t_0]$, for any $t \in (t_1,t_0]$ we have
\[
    |x(t)| \leq |x_0|+B \int_t^{t_0} \frac{\mu(\tau)}{\tau}d\tau
    = |x_0|+ B(\varphi(t_0)-\varphi(t))
\]
and so by the condition $1/r \geq B$ we have
\begin{align*}
   \varphi(t)/r+|x(t)|
   &\leq \varphi(t)/r+|x_0|+ B(\varphi(t_0)-\varphi(t)) \\
   &= \varphi(t_0)/r+|x_0|-(1/r-B)(\varphi(t_0)-\varphi(t)) \\
   &\leq \varphi(t_0)/r+|x_0| <R.
\end{align*}
This shows that $x(t)$ can be extended to $[0,t_0]$. This
proves (1). The result (2) follows from (1). 
\end{proof}

   As to (\ref{6.1.1}) we have

\begin{llem}\label{Lemma6.1.2}
    Suppose the conditions {\rm (\ref{6.1.2})} and {\rm (\ref{6.1.3})}. 
Then, we have the following results.
\par
   {\rm (1)} For any holomorphic function $\psi(x)$ on $D_R$ the 
equation {\rm (\ref{6.1.1})} has a unique solution $v(t,x) \in X_1(W)$ 
satisfying $v(0,x)=\psi(x)$ on $D_R$, and it is given by
\begin{equation}\label{6.1.4}
      v(t,x)= \psi(\phi(0;t,x))
          \exp \Bigl[ \int_0^t
     \frac{\lambda(s,\phi(s;t,x))}{s}ds \Bigr].
\end{equation}
We write this solution as $V_0(\psi)(t,x)$.
\par
   {\rm (2)} If $v(t,x) \in X_1((0,T_1] \times D_{R_1})$ is a solution 
of {\rm (\ref{6.1.1})} we have $v(t,x) \in X_1([0,T] \times D_{R})$ 
for some $T>0$ and $R>0$. By setting $\psi(x)=v(0,x)$ we have 
$v=V_0(\psi)$ on $[0,\epsilon] \times D_{\delta}$ for
some $\psi(x) \in {\mathcal O}_0$, $\epsilon>0$ and $\delta>0$.
\end{llem}

\begin{proof}
    First, let us show (1). Let $v(t,x) \in X_1(W)$ be 
a solution of (\ref{6.1.1}) satisfying $v(0,x)=\psi(x)$ on $D_R$. 
Take any $x_0 \in D_R$.  Set
\begin{align*}
   &v^*(t)=v(t,\phi(t;0,x_0)), \\
   &\lambda^*(t)= \lambda(t,\phi(t;0,x_0)).
\end{align*}
Then, our equation (\ref{6.1.1}) is written in the form
\[
     t \frac{dv^*}{dt} - \lambda^*(t)v^*= 0, \quad 
          v^*(0)= \psi(x_0).
\]
By integrating this, we have
\[
      v^*(t)= \psi(x_0)
          \exp \Bigl[ \int_0^t
     \frac{\lambda^*(s)}{s}ds \Bigr],
\]
that is, 
\[
      v(t,\phi(t;0,x_0))= \psi(x_0)
          \exp \Bigl[ \int_0^t
     \frac{\lambda(s,\phi(s;0,x_0))}{s}ds \Bigr].
\]
\par
   Take any $(t,x) \in W$; by Lemma \ref{Lemma6.1.1} we have a unique 
$x_0 \in D_R$ such that $x=\phi(t;0,x_0)$.
In this case, we have $x_0=\phi(0;t,x)$ and 
$\phi(s;0,x_0)=\phi(s;t,x)$ for $0 \leq s \leq t$. Therefore,
by replacing $x_0$ and $\phi(s;0,x_0)$ by $\phi(0;t,x)$ and 
$\phi(s;t,x)$, respectively, in the above formula we finally obtain
\[
      v(t,x)= \psi(\phi(0;t,x))
          \exp \Bigl[ \int_0^t
     \frac{\lambda(s,\phi(s;t,x))}{s}ds \Bigr].
\]
This proves the formula (\ref{6.1.4}). This shows also the uniqueness
of the solution.
\par
   To show the existence of a solution, it is enough to see that
$v(t,x)$ defined by (\ref{6.1.4}) is a desired solution. Since 
(\ref{6.1.2}) is supposed, we have the well-definedness of $v(t,x)$ 
as a function in $X_1(W)$ and $v(0,x)=\psi(x)$ on $D_R$. 
Thus, the result (1) is proved.
\par
   Next, let us show (2). Let $v(t,x) \in X_1((0,T_1] \times D_{R_1})$ 
be a solution of (\ref{6.1.1}).  Take $0<T<T_1$ and $0<R <R_2<R_1$ so 
that
\[
    R+B\varphi(T)< R_2, \quad R_2+B\varphi(T) <R_1.
\]
Set 
\[
    {\mathcal D}= \bigcup_{x_0 \in D_{R_2}}
      \{(t, \phi(t;T,x_0)) \,;\, 0<t \leq T \}.
\]
It is easy to see that 
$
     (0,T] \times D_R \subset {\mathcal D}
           \subset (0,T] \times D_{R_1}
$
and that ${\mathcal D} \cap \{t=t_0 \}$ is an open subset of $\BC$ 
for any $t_0 \in (0,T]$.  By Proposition \ref{Proposition4.2}
we have the expression
\[
   v(t,x)= v(T,(\phi(T;t,x)) \exp \Bigl[ -\int_{t}^{T}
     \frac{\lambda(s,\phi(s;t,x))}{s}ds \Bigr]
        \quad \mbox{on ${\mathcal D}$}.
\]
Since (\ref{6.1.2}) is supposed, by using this expression we can 
see that
\[
   \sup_{x \in D_R}|v(t_1,x)-v(t_2,x)|
           \longrightarrow 0 \quad 
                 \mbox{(as $t_1,t_2 \longrightarrow +0$)}.
\]
This concludes that $v(t,x)$ converges to a holomorphic 
function $\psi(x)$ (as $t \longrightarrow +0$) uniformly
on $D_R$. Hence, by setting $v(0,x)=\psi(x)$ we have the 
condition 
$v(t,x) \in X_1([0,T] \times D_R)$. This proves the former 
half of (2). The latter half of (2) follows from the uniqueness
of the solution in (1).
\end{proof}

%
% --------------------------------------------------
%  << 6.2. Singular Cauchy problem for (6.1.1) >>
% --------------------------------------------------
%

\subsection{Singular Cauchy problem for (\ref{6.1.1})}
                 \label{subsection6.2}
%{\bf 6.2. Singular Cauchy problem for (6.1.1)}
%

   Let us consider the same equation (\ref{6.1.1}) in the case 
$\lambda(0,x) \not\equiv 0$ on $D_{R_0}$. In this case, instead of 
the Cauchy data $v(0,x)=\psi(x)$ we impose the following singular 
Cauchy data:
\[
     \lim_{t \to +0} t^{-\lambda(0,x)}v(t,x)=\psi(x).
\]
To do so, instead of (\ref{6.1.2}) and (\ref{6.1.3}) we 
suppose: $0<T_0<1/e$, 
\begin{align}
   &|\lambda(t,x)-\lambda(0,x)| \leq \Lambda \mu(t) \quad 
              \mbox{on $[0,T_0] \times D_{R_0}$}, \label{6.2.1}\\
   &|b(t,x)| \leq \frac{B \mu(t)}{|\log t|} \quad 
              \mbox{on $[0,T_0] \times D_{R_0}$} \label{6.2.2}
\end{align}
for some $\Lambda>0$, $B>0$ and a weight function $\mu(t)$
on $(0,T_0]$. We note that (\ref{6.2.2}) implies (\ref{6.1.3}).
\par
   For $T>0$, $R>0$ and $r>0$ we set
$W=\{(t,x) \in [0,T] \times D_R \,;\,
           \varphi(t)/r+|x|<R \}$, as before.

\begin{llem}\label{Lemma6.2.1}
    Suppose {\rm (\ref{6.2.1})}, {\rm (\ref{6.2.2})} and
${\rm Re}\lambda(0,0)>0$. Then, we have the following results.
\par
   {\rm (1)} There are $T>0$, $R>0$ and $r>0$ which satisfy the 
following: for any holomorphic function $\psi(x)$ on $D_R$ the 
equation {\rm (\ref{6.1.1})} has a unique solution $v(t,x) \in X_1(W)$ 
satisfying
\[
      t^{-\lambda(0,x)}v(t,x) \longrightarrow \psi(x)
     \quad \mbox{{\rm (}as $t \longrightarrow +0${\rm )}}
\]
uniformly on $D_{R'}$ for any $0<R'<R$. In addition, we have
$v(t,x) \in X_1^d$ for any $0<d<{\rm Re}\lambda(0,0)$.
We write this solution as $V(\psi)(t,x)$.
\par
   {\rm (2)} If $v(t,x) \in X_1((0,T_1] \times D_{R_1})$ is a 
solution of {\rm (\ref{6.1.1})}, we have $v=V(\psi)$ on 
$(0,\epsilon] \times D_{\delta}$ 
for some $\psi(x) \in {\mathcal O}_0$, $\epsilon>0$ and $\delta>0$.
\end{llem}

\begin{proof}
   We set $w(t,x)=t^{-\lambda(0,x)}v(t,x)$. Then, (\ref{6.1.1}) is 
reduced to an equation with respecto to $w(t,x)$:
\begin{equation}\label{6.2.3}
     t \frac{\partial w}{\partial t}
       = \lambda_b(t,x)w + b(t,x) \frac{\partial w}{\partial x}
\end{equation}
where
\[
   \lambda_b(t,x)= (\lambda(t,x)-\lambda(0,x))
         + \frac{\partial \lambda(0,x)}{\partial x}
         (\log t)b(t,x).
\]
Since $0<T_0<1/e$, (\ref{6.2.1}) and (\ref{6.2.2}) are supposed, 
we have
\begin{align*}
   &|\lambda_b(t,x)| \leq \Lambda \mu(t) + CB\mu(t)
                \quad \mbox{on $[0,T_0] \times D_{R_1}$}, \\
   &|b(t,x)| \leq B\mu(t)/|\log t| \leq M\mu(t)
                      \quad \mbox{on $[0,T_0] \times D_{R_1}$}.
\end{align*}
for some $C>0$ and $0<R_1<R_0$. Hence, we can apply 
Lemma \ref{Lemma6.1.2} to (\ref{6.2.3}). This leads us to 
Lemma \ref{Lemma6.2.1}.
\end{proof}

%
% --------------------------------------------------
%  << 6.3. Singular Cauchy problem for (6.0.1) >>
% --------------------------------------------------
%

\subsection{Singular Cauchy problem for (\ref{6.0.1})}
           \label{subsection6.3}
%{\bf 6.3. Singular Cauchy problem for (6.0.1)}
%

   Let us consider the singular Cauchy problem for (\ref{6.0.1}). 
We have

\begin{tthm}\label{Theorem6.3.1}
    Suppose {\rm (\ref{6.2.1})}, {\rm (\ref{6.2.2})} and
${\rm Re}\lambda(0,0)>0$. Then, for any 
$\psi(x) \in {\mathcal O}_0$ the equation {\rm (\ref{6.0.1})} has
a unique solution $u(t,x) \in X^+_1$ satiasfying
\begin{equation}\label{6.3.1}
      t^{-\lambda(0,x)}u(t,x) \longrightarrow \psi(x)
     \quad \mbox{{\rm (}as $t \longrightarrow +0${\rm )}}
\end{equation}
uniformly on $D_{R}$ for some $R>0$. We write this solution
as $U(\psi)(t,x)$.
\end{tthm}

\begin{proof}
    We take $0<T_1<T_0$, $0<R_1<R_0$, $d>0$, $a_1>a_0>0$ so that
\begin{equation}\label{6.3.2}
    0<d < a_0 \leq {\rm Re}\lambda(t,x) \leq a_1<2d
     \quad \mbox{on $[0,T_1] \times D_{R_1}$}.
\end{equation}

    Step 1:  Let us look for a solution $u(t,x)$ in the form
\[
     u(t,x)=v(t,x)+w(t,x)
\]
where $v(t,x)$ and $w(t,x)$ are defined by the recurrence formulas:
\begin{align}
   &Lv = 0, \label{6.3.3} \\
   &Lw = R \Bigl( t,x,v+w, \frac{\partial v}{\partial x}
              + \frac{\partial w}{\partial x} \Bigr).
             \label{6.3.4}
\end{align}
By Lemma \ref{Lemma6.2.1} we have a solution 
$v(t,x) \in X_1^d([0,T_2] \times D_{R_2})$ of (\ref{6.3.3}) (for some 
$0<T_2<T_1$ and $0<R_2<R_1$) satisfying 
\[
       t^{-\lambda(0,x)}v(t,x) \longrightarrow \psi(x) \quad
    \mbox{{\rm (}as $t \longrightarrow +0${\rm )}}
\]
uniformly on $D_{R_2}$. Then, $v(t,x)$ is a known function, 
and under the setting
\[
     f(t,x)= R \Bigl( t,x,v, \frac{\partial v}{\partial x}\Bigr)
\]
our equation (\ref{6.3.4}) is written in the form
\begin{equation}\label{6.3.5}
   Lw=f(t,x)+ R \Bigl( t,x,v+w, \frac{\partial v}{\partial x}
              + \frac{\partial w}{\partial x} \Bigr)
        - R \Bigl( t,x,v, \frac{\partial v}{\partial x} \Bigr).
\end{equation}
Since $v(t,x) \in X_1^d([0,T_2] \times D_{R_2})$ is valid, 
we have
\[
    |f(t,x)| \leq F t^{2d} \quad 
            \mbox{on $[0,T_2] \times D_{R_2}$}
\]
for some $F>0$. Since $2d>a_1$ is supposed in (\ref{6.3.2}), by
the same argument as in the proof of Theorem \ref{Theorem4.1} (in 
Step 2 in subsection \ref{subsection4.2}) we have a solution 
$w(t,x) \in X_1^{(2d)}$ of (\ref{6.3.5}). Since $2d>a_1$ holds, 
we have 
\[
       t^{-\lambda(0,x)}w(t,x) \longrightarrow 0 \quad
    \mbox{{\rm (}as $t \longrightarrow +0${\rm )}}.
\]
Thus, by setting $u(t,x)=v(t,x)+w(t,x)$ we have a solution 
of (\ref{6.0.1}) satisfying (\ref{6.3.1}). This proves the existence
part of Theorem \ref{Theorem6.3.1}.

    Step 2:  Next, let us show the uniqueness of the solution.
Let $u_i(t,x) \in X_1^+$ ($i=1,2$) be two solutions of (\ref{6.0.1})
satisfying (\ref{6.3.1}). By (\ref{6.3.1}) and (\ref{6.3.2}) we have 
$u_i(t,x) \in X_1^d$ ($i=1,2$). Set 
$w(t,x)=u_1(t,x)-u_2(t,x) \in X_1^d$. Then, we have
\[
     Lw= R \Bigl( t,x,u_2+w, \frac{\partial u_2}{\partial x}
              + \frac{\partial w}{\partial x} \Bigr)
        - R \Bigl( t,x,u_2, \frac{\partial u_2}{\partial x} \Bigr)
\]
and $t^{-\lambda(0,x)}w(t,x) \longrightarrow 0$ 
(as $t \longrightarrow +0$) uniformly near $x=0$.
Since $u_2(t,x)$ and $w(t,x)$ are known function, we have the
expression
\[
     R\Bigl( t,x,u_2+w, \frac{\partial u_2}{\partial x}
              + \frac{\partial w}{\partial x} \Bigr)
        - R \Bigl( t,x,u_2, \frac{\partial u_2}{\partial x} \Bigr)
    =h_1(t,x)w + h_2(t,x) \frac{\partial w}{\partial x}
\]
for some $h_i(t,x) \in X_0([0,T] \times D_R)$ (with $T>0$,
$R>0$), and we have
\[
    |h_i(t,x)| \leq H_i t^d \quad
    \mbox{on $[0,T] \times D_R$}
\]
for some $H_i>0$ ($i=1,2$). Then, under the setting
\[
   L_h= t \frac{\partial}{\partial t}
     -(\lambda(t,x)+h_1(t,x))w-(b(t,x)+h_2(t,x))
        \frac{\partial }{\partial x}
\]
we have
\begin{equation}\label{6.3.6}
    L_hw=0, \quad t^{-\lambda(0,x)}w(t,x) \longrightarrow 0
     \quad \mbox{{\rm (}as $t \longrightarrow +0${\rm )}}.
\end{equation}
By applying the uniqueness part of Lemma \ref{Lemma6.2.1} 
(with $L$ and $\mu(t)$ replaced by $L_h$ and 
$\mu_0(t)=\mu(t)+|\log t|t^d$) to (\ref{6.3.6}) we have $w(t,x)=0$ 
on $[0,\epsilon] \times D_{\delta}$ for some $\epsilon>0$ and 
$\delta>0$. Hence, we have $u_1(t,x)=u_2(t,x)$ on 
$[0,\epsilon] \times D_{\delta}$. 
\par
   This proves the uniqueness of the solution.
\end{proof}

%
% --------------------------------------------------
%  << 6.4. On ${\mathcal S}((\ref{6.0.1}),X_1^+)$ >>
% --------------------------------------------------
%

\subsection{On ${\mathcal S}((\ref{6.0.1}),X_1^+)$}
           \label{subsection6.4}
%{\bf 6.4. On ${\mathcal S}((\ref{6.0.1}),X_1^+)$}
%

   We denote by ${\mathcal S}((\ref{6.0.1}),X_1^+)$ the set of all 
solutions of (\ref{6.0.1}) belonging to the class $X_1^+$.
By Theorem \ref{Theorem6.3.1} we have
\[
    {\mathcal S}((\ref{6.0.1}),X_1^+)
    \supset \{U(\psi) \,;\, \psi(x) \in {\mathcal O}_0 \}.
\]
Let us show

\begin{tthm}\label{Theorem6.4.1}
    Suppose {\rm (\ref{6.2.1})}, {\rm (\ref{6.2.2})} and 
${\rm Re}\lambda(0,0)>0$.  Then, we have
\begin{equation}\label{6.4.1}
    {\mathcal S}((\ref{6.0.1}),X_1^+)
    = \{U(\psi) \,;\, \psi(x) \in {\mathcal O}_0 \}.
\end{equation}
\end{tthm}

\begin{proof}
   Take any $u(t,x) \in {\mathcal S}((\ref{6.0.1}),X_1^+)$.
We have $u(t,x) \in X_1^{\mu}([0,T_1] \times D_{R_1})$
for some $\mu>0$, $T_1>0$ and $R_1>0$. Without loss of 
generality, we may assume that $0<\mu< {\rm Re}\lambda(0,0)$
holds.
\par
   1) If ${\rm Re}\lambda(0,0)<2\mu$, we set $d_1=\mu$ and $N=1$.
\par
   2) If $2\mu \leq {\rm Re}\lambda(0,0)$, we take $d_i>0$
($i=1,\ldots,N$) so that
\[
    0<d_1< \mu<d_2< \cdots<d_N< {\rm Re}\lambda(0,0)
             <2d_N
\]
and that $d_{i+1} \leq 2d_i$ ($i=1,\ldots,N-1$) hold.

    Step 1:  In the case 2) we have $u(t,x) \in X_1^{d_2}$. 
The proof is as follows.
\par
    By taking $a_0>0$, $T_1>0$ and $R_1>0$ suitably, we have
the condition
\[
    \mu<d_2< a_0 \leq {\rm Re}\lambda(t,x)
     \quad \mbox{on $[0,T_1] \times D_{R_1}$}.
\]
We set
\begin{equation}\label{6.4.2}
     f(t,x)= 
    R \Bigl( t,x,u, \frac{\partial u}{\partial x} \Bigr):
\end{equation}
we have $f(t,x) \in X_0([0,T_1] \times D_{R_1})$.
Since $u \in X_1^{\mu}([0,T_1] \times D_{R_1})$ and $d_2<2\mu$ hold,
we have the estimate $|f(t,x)| \leq Ft^{d_2}$ on 
$[0,T_1] \times D_{R_1}$ for some $F>0$. By applying 
Proposition \ref{Proposition4.2} (and Corollary \ref{Corollary4.3}) 
to $Lw=f(t,x)$ we have a solution $w(t,x) \in X_1^{d_2}$ of $Lw=f(t,x)$. 
Since $u(t,x)$ is a solution of $Lu=f(t,x)$, by setting 
$v(t,x)=u(t,x)-w(t,x)$ we have an equation
\[
          Lv=0
\]
and $v(t,x) \in X_1^{\mu}$. Then, by (2) of Lemma \ref{Lemma6.2.1} we 
have an expression $v=V(\psi)$ for some $\psi(x) \in {\mathcal O}_0$,
that is, we have the condition:
\[
      t^{-\lambda(0,x)}v(t,x) \longrightarrow \psi(x)
     \quad \mbox{(as $t \longrightarrow +0$)}
\]
uniformly on $D_{R}$ for some $R>0$. This means that $v \in X_1^{(a_0)}$.
Hence, we have
\[
    u(t,x)=v(t,x)+w(t,x) \in X_1^{(a_0)}+X_1^{d_2}
         \subset X_1^{d_2}.
\]

    Step 2:  In the case 1) we know $u(t,x) \in X_1^{d_N}$ from 
the first. In the case 2), by Step 1 we have $u(t,x) \in X_1^{d_2}$:
then, by applying the same argument we have $u(t,x) \in X_1^{d_3}$,
and so by repeating the same argument we have finally the condition 
$u(t,x) \in X_1^{d_N}$. Thus, in any case we have 
$u(t,x) \in X_1^{d_N}$.
\par
   Then, the function
$f(t,x)$ in (\ref{6.4.2}) satisfies $|f(t,x)| \leq F_1t^{2d_N}$ on
$[0,T_2] \times D_{R_2}$ for some $T_2>0$ and $R_2>0$. Since 
$2d_N > {\rm Re}\lambda(0,0)$ holds, by applying 
Theorem \ref{Theorem3.1} to $Lw=f$ we have a solution 
$w(t,x) \in X_1^{(2d_N)}$ of $Lw=f$. Since $u(t,x)$ is a solution of 
$Lu=f(t,x)$, by setting $v(t,x)=u(t,x)-w(t,x)$ we have an equation
\[
          Lv=0
\]
and $v(t,x) \in X_1^{d_N}$. Hence, by (2) of Lemma \ref{Lemma6.2.1} we 
have an expression $v=V(\psi)$ for some $\psi(x) \in {\mathcal O}_0$,
that is, we have the condition:
\[
      t^{-\lambda(0,x)}v(t,x) \longrightarrow \psi(x)
     \quad \mbox{(as $t \longrightarrow +0$)}
\]
uniformly on $D_{R}$ for some $R>0$. Since $w(t,x) \in X_1^{(2d_N)}$ 
and $2d_N>{\rm Re}\lambda(0,0)$ hold, we have 
$t^{-\lambda(0,x)}w(t,x) \longrightarrow 0$
(as $t \longrightarrow +0$) uniformly near $x=0$. 
Since $u(t,x)=v(t,x)+w(t,x)$, we obtain
\[
      t^{-\lambda(0,x)}u(t,x) \longrightarrow \psi(x)
     \quad \mbox{(as $t \longrightarrow +0$)}
\]
uniformly near $x=0$.
Thus, by the uniqueness of the solution in Theorem \ref{Theorem6.3.1} 
we have $u(t,x)=U(\psi)(t,x)$ on $[0,\epsilon] \times D_{\delta}$
for some $\epsilon>0$ and $\delta>0$.
\par
   This proves (\ref{6.4.1}).
\end{proof}

%
% --------------------------------------------------
%  << 6.5. On ${\mathcal S}((\ref{6.0.1}),{\mathscr X}_1)$ >>
% --------------------------------------------------
%

\subsection{On ${\mathcal S}((\ref{6.0.1}),{\mathscr X}_1)$}
        \label{subsection6.5}
%{\bf 6.5. On ${\mathcal S}((\ref{6.0.1}),{\mathscr X}_1)$}
%

   Lastly, we consider (\ref{6.0.1}) under the 
conditions (\ref{6.1.3}) and ${\rm Re} \lambda(0,0)>0$. The purpose 
is to prove

\begin{tthm}\label{Theorem6.5.1}
    Under the conditions {\rm (\ref{6.1.3})} and 
${\rm Re} \lambda(0,0)>0$, we have
\begin{equation}\label{6.5.1}
     {\mathcal S}((\ref{6.0.1}),{\mathscr X}_1)
                  = {\mathcal S}((\ref{6.0.1}),X_1^+).
\end{equation}
\end{tthm}

\begin{proof}
    Since
${\mathcal S}((\ref{6.0.1}),{\mathscr X}_1)
                  \supset {\mathcal S}((\ref{6.0.1}),X_1^+)$
is trivial, it is enough to prove that 
${\mathcal S}((\ref{6.0.1}),{\mathscr X}_1)
                  \subset {\mathcal S}((\ref{6.0.1}),X_1^+)$
holds. Let us show this from now.
\par
   The argument below is quite similar to the one used in the proof
of [Theorem 2.2 in \cite{unique}], where we have treated the case
${\rm Re}\lambda(0,0)<0$. Since we are now considering the case
${\rm Re}\lambda(0,0)>0$, the treating of some parts must be changed.

    Step 1:  Take any $u(t,x) \in 
{\mathcal S}((\ref{6.0.1}),{\mathscr X}_1)$; then we have
$u(t,x) \in X_1((0,T_1] \times D_{R_1})$ and 
\begin{equation}\label{6.5.2}
     \varlimsup_{R \to +0}  \biggl[ \, 
          \lim_{T \to +0}  \Bigl( \frac{1}{R^2}
       \sup_{(0,T] \times D_R}|u(t,x)| \Bigr)  \biggr]
        =  0.  
\end{equation}
Since $u(t,x)$ is a solution of (\ref{6.0.1}), under setting
$q(t,x)=(\partial u/\partial x)(t,x)$ we have the relation
\[
   t \,\frac{\partial u}{\partial t}
   = \lambda(t,x)u+ b(t,x)\frac{\partial u}{\partial x} 
     + R_1(t,x,u,q)u+R_2(t,x,u,q)\frac{\partial u}{\partial x}
\]
on $(0,T_1] \times D_{R_1}$ for some $R_i(t,x,z_1,z_2) \in C^0(\Omega)$
($i=1,2$) that are holomorphic in the variable 
$(x,z_1,z_2) \in D_{R_0} \times D_{\rho_0} \times D_{\rho_0}$ 
for any fixed $t$ and that satisfy $R_i(t,x,0,0) \equiv 0$ ($i=1,2$).
\par
   By taking $T_1>0$ and $R_1>0$ sufficiently small, we may 
assume that
\begin{align*}
   &{\rm Re} \lambda(t,x)>2a>0, 
        \quad \mbox{on $[0,T_1] \times D_{R_1}$}, \\
   &|R_1(t,x,u,q)| \leq A_1|u|+A_2|q|,
               \quad \mbox{on $[0,T_1] \times D_{R_1}$}, \\
   &|R_2(t,x,u,q)| \leq B_1|u|+B_2|q|,
               \quad \mbox{on $[0,T_1] \times D_{R_1}$}
\end{align*}
for some $a>0$, $A_1>0$, $A_2>0$, $B_1>0$ and $B_2>0$.
By setting 
\begin{align*}
   &\lambda_1(t,x)=R_1(t,x,u(t,x),q(t.x)) 
                          \in X_0((0,T_1] \times D_{R_1})), \\
   &b_1(t,x)=R_2(t,x,u(t,x),q(t.x)) 
                          \in X_0((0,T_1] \times D_{R_1}))
\end{align*}
we have the relation
\begin{equation}\label{6.5.3}
   t \,\frac{\partial u}{\partial t}-(b(t,x)+b_1(t,x))
         \frac{\partial u}{\partial x}
   = (\lambda(t,x)+\lambda_1(t,x))u 
\end{equation}
on $(0,T_1] \times D_{R_1}$. By applying $\partial/\partial x$
to (\ref{6.5.3}) we have
\begin{align}
   &t \,\frac{\partial q}{\partial t}-(b(t,x)+b_1(t,x))
         \frac{\partial q}{\partial x} \label{6.5.4} \\
   &\qquad \qquad = (\lambda(t,x)+\lambda_1(t,x)+c(t,x))q 
        + \gamma(t,x)u \notag
\end{align}
on $(0,T_1] \times D_{R_1}$, where
\begin{align*}
   &c(t,x)= \frac{\partial b}{\partial x}(t,x)+ 
                  \frac{\partial b_1}{\partial x}(t,x) 
             \in X_0((0,T_1] \times D_{R_1}), \\
   &\gamma(t,x)= \frac{\partial \lambda}{\partial x}(t,x)+ 
                  \frac{\partial \lambda_1}{\partial x}(t,x) 
             \in X_0((0,T_1] \times D_{R_1}).
\end{align*}

    Step 2:  For $0<T<T_1$ and $0<R<R_1$ we set
\begin{align*}
   &r_1 = \sup_{(0,T] \times D_R}|u(t,x)|, \quad
   r_2 = \sup_{(0,T] \times D_R}|q(t,x)|, \\
   &C= \sup_{(0,T] \times D_R}|c(t,x)|, \quad
    L= \sup_{(0,T] \times D_R}|\gamma(t,x)|.
\end{align*}

\begin{llem}\label{Lemma6.5.2}
    We can take $0<T<T_1$ and $0<R<R_1$ 
so that the following conditions are satisfied:
\begin{align}
   &A_1r_1+A_2r_2 \leq \frac{a}{2}, \label{6.5.5} \\
   &B\varphi(T)+ \frac{B_1r_1+B_2r_2}{a}
            + \frac{B_2Lr_1}{a^2}< \frac{R}{4}, \label{6.5.6} \\
   &C \leq \frac{a}{2}. \label{6.5.7}
\end{align}
\end{llem}

\begin{proof}
    By (\ref{6.5.2}) we have
\begin{align*}
   &\lim_{T \to 0}r_1 = o(R^2) 
             \quad \mbox{(as $R \longrightarrow +0$)}, \\
   &\lim_{T \to 0}r_2 = o(R) 
             \quad \mbox{(as $R \longrightarrow +0$)}
\end{align*}
and so by taking $T>0$ and $R>0$ sufficiently small we have
the conditions (\ref{6.5.5}) and (\ref{6.5.6}). Since 
$|b_1(t,x)| \leq B_1r_1+B_2r_2$ on $(0,T] \times D_R$ we have
\begin{align*}
   &\lim_{T \to 0}\sup_{(0,T] \times D_R}|b_1(t,x)| 
          = o(R) \quad \mbox{(as $R \longrightarrow +0$)}, \\
   &\lim_{T \to 0}\sup_{(0,T] \times D_R}|(\partial b_1/\partial x)(t,x)| 
          = o(1) \quad \mbox{(as $R \longrightarrow +0$)}.
\end{align*}
Therefore, by taking $T>0$ and $R>0$ sufficiently small
we have the condition $|(\partial b_1/\partial x)(t,x)| \leq a/4$ 
on $(0,T] \times D_R$.
Since $|b(t,x)| \leq B\mu(t)$ holds on $[0,T_0] \times D_{R_0})$,
by taking $T>0$ sufficiently small we have 
$|(\partial b/\partial x)(t,x)| \leq a/4$ on $(0,T] \times D_R$.
Since $c(t,x)= (\partial b_1/\partial x)(t,x)
          + (\partial b/\partial x)(t,x)$ we have the condition 
(\ref{6.5.7}). 
\end{proof}

\begin{ccor}\label{Corollary6.5.3}
   Under the situation in Lemma {\rm \ref{Lemma6.5.2}} we have
\begin{align*}
    &{\rm Re}(\lambda(t,x)+\lambda_1(t,x)) \geq a
           \quad \mbox{on $(0,T] \times D_R$}, \\
    &{\rm Re}(\lambda(t,x)+\lambda_1(t,x)+c(t,x)) \geq a
           \quad \mbox{on $(0,T] \times D_R$}.
\end{align*}
\end{ccor}

\begin{proof}
    Since $|\lambda_1(t,x)| \leq A_1r_1+A_2r_2$ 
and $|c(t,x)| \leq C$ hold on $(0,T] \times D_R$, by (\ref{6.5.5})
and (\ref{6.5.7}) we have the result.
\end{proof}

    Step 3:  From now, we take $T>0$ and $R>0$ as in 
Lemma \ref{Lemma6.5.2}, and fix them. For $\xi \in D_R$, let us 
consider the initial value problem
\begin{equation}\label{6.5.8}
    \frac{dx}{dt}= - \frac{b(t,x)+b_1(t,x)}{t},
       \quad x \bigr|_{t=T}=\xi. 
\end{equation}
Let $(t_{\xi},T]$ be the maximal existence domain of the solution
$x(t)$ of (\ref{6.5.8}). Set
\begin{align*}
    &u^*(t)= u(t,x(t)), \quad q^*(t)= q(t,x(t)), \\
    &\lambda^*(t)= \lambda(t,x(t)), 
               \quad \lambda_1^*(t)= \lambda_1(t,x(t)), \\
    &c^*(t)= c(t,x(t)), \quad
    \gamma^*(t)= \gamma(t,x(t)):
\end{align*}
these functions are well-defined on $(t_{\xi},T]$. By 
(\ref{6.5.3}) and (\ref{6.5.4}) we have
\begin{align}
   &t \frac{du^*}{dt}= (\lambda^*(t)+\lambda_1^*(t))u^*
        \quad \mbox{on $(t_{\xi},T]$}, \label{6.5.9}\\
   &t \frac{dq^*}{dt}= (\lambda^*(t)+\lambda_1^*(t)+c^*(t))q^*
                + \gamma^*(t)u^*
        \quad \mbox{on $(t_{\xi},T]$}. \label{6.5.10}
\end{align}

\begin{llem}\label{Lemma6.5.4}
   In the above situation, for any $t \in (t_{\xi},T]$ we have
\begin{align}
    &|u^*(t)| \leq r_1 \Bigl(\frac{t}{T} \Bigr)^a, 
              \label{6.5.11} \\
    &|q^*(t)| \leq r_2  \Bigl(\frac{t}{T} \Bigr)^a
             + Lr_1 \Bigl(\frac{t}{T} \Bigr)^a 
                \log\Bigl(\frac{T}{t} \Bigr), \label{6.5.12} \\
    &|x(t)| \leq |\xi|+ B\varphi(T)+ \frac{B_1r_1+B_2r_2}{a}
            + \frac{B_2Lr_1}{a^2}. \label{6.5.13}
\end{align}
\end{llem}

\begin{proof}
     We set
\[
     E(t)= \exp \Bigl[ \int_t^T 
              \frac{\lambda^*(s)+\lambda_1^*(s)}{s} ds \Bigr],
       \quad t_{\xi}<t \leq T.
\]
Then, by (\ref{6.5.9}) we have
\[
     \frac{d}{dt} ( u^*(t)E(t))=0,
            \quad t_{\xi}<t \leq T,
\]
and so by integrating this from $t$ to $T$ we have 
$u^*(T)-u^*(t)E(t)=0$, that is, 
\[
     u^*(t)= u^*(T) \exp \Bigl[ -\int_t^T 
              \frac{\lambda^*(s)+\lambda_1^*(s)}{s} ds \Bigr]
  \quad \mbox{on $(t_{\xi},T]$}.
\]
Since $|u^*(T)| \leq r_1$ and 
${\rm Re}(\lambda^*(t)+\lambda_1^*(t)) \geq a$ on $(t_{\xi},T]$,
we have
\[
    |u^*(t)| \leq r_1 \exp \Bigl[ -\int_t^T 
              \frac{a}{s} ds \Bigr]
         = r_1 \Bigl(\frac{t}{T} \Bigr)^a 
                   \quad \mbox{on $(t_{\xi},T]$}.
\]
This proves (\ref{6.5.11}). Similarly, if we set
\[
     E_1(t)= \exp \Bigl[ \int_t^T 
       \frac{\lambda^*(s)+\lambda_1^*(s)+c^*(s)}{s} ds \Bigr],
       \quad t_{\xi}<t \leq T,
\]
by the same argument we have 
\begin{align*}
    q^*(t)= &q^*(T) \exp \Bigl[ -\int_t^T 
       \frac{\lambda^*(s)+\lambda_1^*(s)+c^*(s)}{s} ds \Bigr]\\
     &- \int_t^T \exp \Bigl[ -\int_t^{\tau}
        \frac{\lambda^*(s)+\lambda_1^*(s)+c^*(s)}{s} ds \Bigr]
             \frac{\gamma^*(\tau)u^*(\tau)}{\tau} d\tau.
\end{align*}
Since $|q^*(T)| \leq r_2$, 
${\rm Re}(\lambda^*(t)+\lambda_1^*(t)+c^*(t)) \geq a$ and
$|\gamma^*(t)| \leq L$ on $(t_{\xi},T]$, 
by (\ref{6.5.11}) we have
\begin{align*}
   |q^*(t)| &\leq r_2 \Bigl(\frac{t}{T} \Bigr)^a
          + \int_t^T \Bigl(\frac{t}{\tau} \Bigr)^a 
                      \frac{Lr_1(\tau/T)^a}{\tau} d\tau \\
     &= r_2 \Bigl(\frac{t}{T} \Bigr)^a 
              + Lr_1 \Bigl(\frac{t}{T} \Bigr)^a 
              \log\Bigl(\frac{T}{t} \Bigr) 
                            \quad \mbox{on $(t_{\xi},T]$}.
\end{align*}
This proves (\ref{6.5.12}). 
\par
   Let us show (\ref{6.5.13}). Since $x(t)$ is a solution of 
(\ref{6.5.8}), we have
\[
    x(t)=\xi + \int_t^T 
     \frac{b(\tau,x(\tau))+b_1(\tau,x(\tau))}{\tau} d\tau.
\]
Since $|b_1(t,x)| \leq B_1|u(t,x)|+B_2|q(t,x)|$ is known on 
$(0,T_1] \times D_{R_1}$, by (\ref{6.1.3}), (\ref{6.5.11}) and 
(\ref{6.5.12}) we have
\begin{equation}\label{6.5.14}
    |b(t,x(t))|+|b_1(t,x(t))| \leq \Phi(t) \quad
        \mbox{on $(t_{\xi},T]$} 
\end{equation}
with
\begin{equation}\label{6.5.15}
    \Phi(t) = B\mu(t)+ B_1r_1 \Bigl(\frac{t}{T} \Bigr)^a 
        +B_2r_2 \Bigl(\frac{t}{T} \Bigr)^a
        + B_2Lr_1 \Bigl(\frac{t}{T} \Bigr)^a 
                           \log\Bigl(\frac{T}{t} \Bigr).
\end{equation}
Hence, we obtain
\begin{align*}
   |x(t)| &\leq |\xi|+ \int_t^T \frac{\Phi(\tau)}{\tau} d\tau \\
      &\leq |\xi|+ B\varphi(T)+ \frac{B_1r_1+B_2r_2}{a}
            + \frac{B_2Lr_1}{a^2} \quad \mbox{on $(t_{\xi},T]$}.
\end{align*}
This proves (\ref{6.5.13}).
\end{proof}

\begin{ccor}\label{Corollary6.5.5}
     If $|\xi|<3R/4$, we have $t_{\xi}=0$.
\end{ccor}

\begin{proof}
    If $|\xi|<3R/4$, by (6.5.6) and (6.5.13) we have
\[
    |x(t)| \leq \frac{3R}{4}+B\varphi(T)+ \frac{B_1r_1+B_2r_2}{a}
            + \frac{B_2Lr_1}{a^2}<R
\]
on $(t_{\xi},T]$. If $t_{\xi}>0$, this estimate asserts that the
solution $x(t)$ can be extended to $(t_{\xi}-\delta,T]$ for some
$\delta>0$. But, this contradicts the condition that $(t_{\xi},T]$
is a maximal existence domain. Hence, we have $t_{\xi}=0$.
\end{proof}

    Step 4:  We denote by $\phi_1(t;T,\xi)$ the unique solution of
(\ref{6.5.8}): we have
\[
    \phi_1(t;T,\xi)
    = \xi + \int_t^T \frac{b(\tau,\phi_1(\tau;T,\xi))
          +b_1(\tau,\phi_1(\tau;T,\xi))}{\tau}d\tau.
\]
Set
\[
   {\mathcal D}= \bigcup_{\xi \in D_{3R/4}}
     \{(t,\phi_1(t;T,\xi)) \,;\, 0<t \leq T \}.
\]
Then, by Lemma \ref{Lemma6.5.4} we have
\begin{align*}
    &|u(t,x)| \leq r_1 \Bigl(\frac{t}{T} \Bigr)^a
            \quad \mbox{on ${\mathcal D}$}, \\
    &|q(t,x)| \leq r_2 \Bigl(\frac{t}{T} \Bigr)^a 
             + Lr_1 \Bigl(\frac{t}{T} \Bigr)^a 
                \log\Bigl(\frac{T}{t} \Bigr)
            \quad \mbox{on ${\mathcal D}$}.
\end{align*}
Therefore, to get the condition $u(t,x) \in X_1^+$ it is
enough to prove the following result.

\begin{llem}\label{Lemma6.5.6}
     We have the condition
\[
      (0,T] \times D_{R/4} \subset {\mathcal D}. 
\]
\end{llem}

\begin{proof}
    To show this, it is sufficient to porve
the following assertion: for any $(t_0,x_0) \in (0,T) \times D_{R/4}$ 
there is a unique $\xi \in D_{R/2}$ such that $\phi_1(t_0;T,\xi)=x_0$ 
holds, that is, 
\begin{equation}\label{6.5.16}
    x_0= \xi + \int_{t_0}^T 
        \frac{b(\tau,\phi_1(\tau;T,\xi))+b_1(\tau,\phi_1(\tau;T,\xi))}
             {\tau}d\tau
\end{equation}
holds.  For simplicity, we set
\[
     G(t,\xi)= b(t,\phi_1(t;T,\xi))+b_1(t,\phi_1(t;T,\xi))
     \in X_0((0,T] \times D_{3R/4}).
\]
By (\ref{6.5.14}) and the Cauchy's inequality, we have
\begin{align*}
   &|G(t,\xi)| \leq \Phi(t) \quad 
               \mbox{on $(0,T] \times D_{3R/4}$}, \\
   &\Bigl| \frac{\partial G}{\partial \xi}(t,\xi) \Bigr| 
           \leq \frac{\Phi(t)}{R/4} \quad 
               \mbox{on $(0,T] \times D_{R/2}$}.
\end{align*}
Therefore, we have the Lipschitz condition:
\[
     |G(t,\xi_1)-G(t,\xi_2)| \leq \frac{\Phi(t)}{R/4}|\xi_1-\xi_2|
\]
for any $(t,\xi_i) \in (0,T] \times D_{R/2}$ ($i=1,2$). We set
also
\[
    K=B\varphi(T)+ \frac{B_1r_1+B_2r_2}{a}
            + \frac{B_2Lr_1}{a^2}:
\]
by (\ref{6.5.6}) we have $K<R/4$ and under the notation 
(\ref{6.5.15}) we have
\[
    \int_0^T \frac{\Phi(\tau)}{\tau} d\tau =K.
\]
\par
   Let us solve (\ref{6.5.16}), that is, 
\begin{equation}\label{6.5.17}
    \xi= x_0- \int_{t_0}^T 
        \frac{G(\tau,\xi)}{\tau}d\tau.
\end{equation}
To solve this equation, we use the method of successive 
approximations: the approximate solutions $\xi_n$ ($n=0,1,2,\ldots$)
are defined by $\xi_0=x_0$ and 
\[
    \xi_n= x_0- \int_{t_0}^T 
        \frac{G(\tau,\xi_{n-1})}{\tau}d\tau, \quad n \geq 1.
\]
\par
   Since $\xi_0=x_0 \in D_{R/4}$, we see that $G(\tau,\xi_0)$ is 
well-defined. Then, $\xi_1$ is well-defined and 
\[
    |\xi_1| \leq |x_0|+ \int_{t_0}^T \frac{\Phi(\tau)}{\tau} d\tau
      <R/4+ K <R/4+R/4=R/2.
\]
In addition, we have 
\[
   |\xi_1-\xi_0| \leq 
      \Bigl| \int_{t_0}^T \frac{G(\tau,\xi_0)}{\tau}d\tau \Bigr|
      \leq \int_{t_0}^T \frac{\Phi(\tau)}{\tau}d\tau  = K.
\]
In this way, by a standard argument we can see that 
$\xi_n \in D_{R/4+ K} \subset D_{R/2}$
($n=0,1,2,\ldots$) are well-defined and we have
\[
    |\xi_n-\xi_{n-1}| \leq K \Bigl(\frac{K}{R/4} \Bigr)^{n-1},
    \quad n=1,2,\ldots.
\]
Since $K/(R/4)<1$ is supposed, we see that $\xi_n$ is convergent 
(as $n \longrightarrow \infty$) to a solution $\xi \in D_{R/2}$ 
of (\ref{6.5.17}). This shows the existence of a solution 
$\xi \in D_{R/2}$ of (\ref{6.5.16}).
\par
   The uniqueness of the solution can be proved in the same way.
\end{proof}

    This completes the proof of Theorem \ref{Theorem6.5.1}.
\end{proof}

%
%  --------------------------------
%        << References >>}
%  --------------------------------
%

\end{document}